\documentclass[a4paper,12pt,twoside]{article}
\usepackage{amsfonts}
\usepackage{rotating}
\usepackage{chapterbib}
\usepackage{graphicx, color}
\usepackage[latin1]{inputenc}
\usepackage[english]{babel}
\usepackage{amssymb}
\usepackage{amsmath, eurosym, layout, hyperref} %stmaryrd,
\usepackage{times}
\usepackage{amsthm}
\usepackage{makeidx}

\makeindex
\usepackage{linegoal}
\newcommand{\IN}{\mathbb{N}}

\newcommand{\IE}{\mathbb{E}}
\newcommand{\IP}{\mathbb{P}}

\newcommand{\sP}{\mathcal{P}}

\newcommand{\sD}{\mathcal{D}}
\newcommand{\sA}{\mathcal{A}}

\newcommand{\sV}{\mathcal{V}}

\newtheorem{theorem}{Theorem}[section]
\newtheorem{proposition}[theorem]{Proposition}

\newtheorem{corollary}[theorem]{Corollary}

\theoremstyle{definition}
\newtheorem{remark}[theorem]{Remark}

\newtheorem{notation}[theorem]{Notation}
\newtheorem{definition}[theorem]{Definition}

\newtheorem{examples}[theorem]{Examples}

\newcommand{\bew}{\noindent\textbf{Proof:}\quad}
\newcommand{\ebew}{\hfill\qed\\}

\newsavebox{\prN}

\renewenvironment{proof}{\bew}{\ebew}

\newenvironment{balign*}[1][12pt]{\setlength{\jot}{#1}\nonumber\align}{\endalign}

\renewcommand{\epsilon}{\varepsilon}

\newcommand{\dotsfi}{\quad\leaders\hbox to 25pt{\hss\tiny.\hss}\hfill}

\title{Effectiveness, Decisiveness and Success\\ in Weighted Voting Systems }
\author{Werner Kirsch \\
Fakult\"{a}t f\"{u}r Mathematik und Informatik\\
FernUniversit\"{a}t in Hagen,
Germany}

\begin{document}
\maketitle

\section{Introduction}
The notions effectiveness, decisiveness and success are basic to the analysis of voting systems. Yet, they do not only depend on the voting rule but also on the underlying voting
measure, i.\,e. on the correlation of the voting behavior between the voters of the system. In the  Penrose-Banzhaf case the voting measure gives equal probability to all coalitions, thus
reflecting the situation when each voter's decision is completely independent of the other voters. The corresponding power index (in terms of decisiveness)  is the well known Penrose-Banzhaf index.
Under this voting measure there is a simple formula connecting the power index of a voter with the probability of success of this voter (see \eqref{eq:  DS} below).

As was emphasized in \cite{Laruelle}, this intimate connection between decisiveness and success is a peculiarity of the Penrose-Banzhaf measure. In particular there is no
analog for the Shapley-Shubik power index. The Shapley-Shubik index is based on decisiveness under a voting measure we call the Shapley-Shubik measure. Under this measure all coalitions
of a given size $k$ have the same ($k$-dependent) probability and the set of all coalitions of size $k$ is given a probability \emph{independent} of $k$.

Among others we consider voting systems with `simple voting rule', that is with voting weight $1$ for all voters, but with arbitrary relative quota $r$.
For such systems we compute the probability of decisiveness and the rate of success under the Penrose-Banzhaf measure and the Shapley-Shubik measure. One of the results which
we found surprising is that the rate of success under the Shapley-Shubik measure is (approximately) $\frac{3}{4}$ in the case of simple majority (i.\,e. $r=\frac{1}{2}$). Thus it is
remarkably bigger than the rate of success under the Penrose-Banzhaf measure. On the other hand the rate of decisiveness in the same situation is bigger under the Penrose-Banzhaf measure.

We extend the mentioned result to general $r$, to weighted voting systems and to more general voting measures.

\section{Some Basics}
IN this section we introduce some of the concepts basic for the rest of this paper. For a thorough introduction we recommend \cite{Felsenthal}, for an overview \cite{Taylor} or \cite{KirschBer}.
\begin{definition}
A $\emph{voting system} ~  (V,\mathcal{V})$ consists of a (finite) set $V$ of voters and a subset $\mathcal{V}$ of $\mathcal{P} (V)$, the system of all subsets of $V$, with the following properties
\begin{enumerate}
\item $V \in \mathcal{V}$
\item $\emptyset \notin \mathcal{V}$
\item If $A \in \mathcal{V}$ and $A \subset B$ then $B \in \mathcal{V}$.
\end{enumerate}
Subsets of $V$ are called $\emph{coalitions}$. The coalitions in $\mathcal{V}$ are called $\emph{winning}$, those not in $\mathcal{V} ~ \emph{losing}$.
\end{definition}
\begin{definition}
A voting system $(V,\mathcal{V})$ is called $\emph{weighted}$ if there is a function $w: V \rightarrow [0, \infty)$ and a $q \in (0, \infty)$ such that $A \in \sV$ if and only if
\begin{align}
\sum_{v \in A} w(v) \geq q
\end{align}
The number $w(v)$ is called the weight of the voter $v$, $q$ is called the $\emph{quota}$. The number
\begin{align}
r = \frac{q}{\sum_{v \in V} w(v)}
\end{align}
is called the $\emph{relative quota}$.\\
We call a weighted voting system $\emph{simple}$, if $w(v) = 1$ for all $v$.\\
A simple voting system $(V,\mathcal{V})$ is called a \emph{simple majority system} if the relative quota $r$ is given by $r = \frac{1}{2} + \frac{1}{2N}$, where $N = |V|$ is the number of voters. In other words, those coalitions are winning which contain more than half of the voters.
\end{definition}
There are various methods to quantify the notion `voting power' in voting systems.\\
One of the best known concepts goes back to Penrose $\cite{Penrose}$ and Banzhaf $\cite{Banzhaf}$. It is based on the notion of `decisiveness' and the treatment of all coalitions as `equally likely'.
\begin{definition}
Suppose $(V,\mathcal{V})$ is a voting system.
\begin{enumerate}
\item We call a voter $v \in V ~  \emph{winning decisive}$ for a coalition $A \subset V$ if $v \notin A, A \notin \mathcal{V}$ and $A\cup \{v\} \in \mathcal{V}$.\\
We set
\begin{align}
\sD^{+} (v) := \{A \subset V \mid v \notin A, A \notin \mathcal{V}, A\cup \{v\} \in \mathcal{V}\}
\end{align}
\item We call a voter $v \in V ~  \emph{losing decisive}$ for a coalition $A \subset V$ if $v \in A, A \in \mathcal{V}, A \setminus \{v\} \notin \mathcal{V}$.\\
We set
\begin{align}
\sD^- (v) := \{ A \subset V \mid v \in A, A \in \mathcal{V}, A \setminus \{v \} \notin \mathcal{V} \}
\end{align}
\item We call $v ~  \emph{decisive}$ for $A$ if $v$ is winning decisive or losing decisive for $A$ and set
\begin{align}
\sD (v) := \sD^+ (v) \cup \sD^- (v)
\end{align}
\end{enumerate}
\end{definition}
\begin{definition}
The $\emph{Penrose-Banzhaf power} ~  PB (v)$ of a voter $v$ is defined as
\begin{align}
D_B (v) = \frac{|\sD (v)|}{2^N}
\end{align}
where $|A|$ denotes the number of elements of the set $A$ and $N = |V|$.
\end{definition}
This $PB (v)$ is the proportion of all coalition for which $v$ is decisive.
\begin{remark}
It is well known and easy to see that $|\sD^+ (v)| = |\sD^- (v)|$, so
\begin{align}
D_B (v) = \frac{|\sD^+ (v)|}{2^{N-1}} = \frac{|\sD^- (v)|}{2^{N-1}}
\end{align}
The Penrose-Banzhaf power admits a prohabilistic interpretation. If we regard all coalitions in $\mathcal{P} (V)$ as equally likely (`Laplace probability') and denote the corresponding measure on $\mathcal{P} (V)$ by
\begin{align}
\IP_B (\sA) := \frac{|\sA|}{|\sP(V)|} = \frac{|\sA|}{2^N}
\end{align}
then
\begin{align}
D_B (v) = \IP_B \Big(\sD (v)\Big)\,.
\end{align}
We call $\IP_B$ the $\emph{Penrose-Banzhaf}$ measure.

Without loss we may assume that $V = \{1, \cdots , N\}$.\\
Instead of considering $\IP_B$ as a measure on $\sP(V)$ we may consider $\IP_B$ as a measure on $\{0,1\}^N (= \{0,1\}^V)$ by
\begin{align}\label{eq:pv01}
\IP_B \Big(\{(x_1, \cdots ,x_N) \}\Big) := \IP_B \Big(\{\{i | x_i = 1\}\}\Big)
\end{align}
In the following we will switch freely between these versions of $\IP_B$. Moreover, to simplify notation we will write $\IP_B \big((x_1, \cdots , x_N)\big)$ instead of $\IP_B \big(\{(x_1, \cdots , x_N)\}\big)$ and $\IP_B (A)$ instead of $\IP_B (\{A\})$ for any $A \in \sP(V)$.\\
In this paper we will introduce and discuss various other measures on $\sP(V)$ resp. $\{0,1\}^N$ which lead to different notions of voting power, for example to the Shapley-Shubik index $\cite{ShapleyS}$.\\
We now introduce this concept in an abstract setting.
\end{remark}
\begin{definition}
A probability measure $\IP$ on $\sP(V)$ (resp. $\{0,1\}^N$) is called a $\emph{voting measure}$ if
\begin{align}
\IP (A) = \IP (V \setminus A) \text{~  for all ~} A \subset V
\end{align}
resp. $\IP\Big (x_1, \cdots ,x_N)\Big) = \IP \Big((1 - x_1, \cdots , 1 - x_N)\Big)$.\\
\end{definition}
The papers \cite{KirschPen,KirschBer} contain a discussion about why this is an appropriate definition.

As in the case of the Penrose-Banzhaf power we may define a voting power in terms of decisiveness by
\begin{align}
D_{\IP}^+ (v) & := \IP (\sD^+ (v))\\
D_{\IP}^- (v) & := \IP (\sD^- (v))\\
\text{and ~} D_{\IP} (v) & := \IP (\sD (v)) = D_{\IP}^+ (v) + D_{\IP}^- (v)
\end{align}
Note, that we destinguish here between $D_{\IP}^+ (v)$ (the probability to make a losing coalition winning) and $D_{\IP}^- (v)$ (the probability to make a winning coalition losing). In contrast to the case of the Penrose-Banzhaf measure we can $\emph{not}$ conclude $\IP \big(\sD^+ (v)\big) = \IP \big(\sD^- (v)\big)$.\\
%In fact, we prove in the Appendix the following proposition.

%\begin{proposition}
%Suppose $\IP$ is a votin measure on $\{0,1\}^N$. If $\IP (\sD^+ (v)) = \IP (\sD^- (v))$ holds for any simple voting system on %$\{1, \cdots , N\}$ then $\IP (A) = \frac{1}{2^N}$.
%\end{proposition}
\begin{examples}\label{ex:vm}
We give some examples for voting measures:
\begin{enumerate}
\item The Penrose-Banzhaf measure.
\item The Shapley-Shubik measure.\\
If $A \subset \sP(V)$ with $|V| = N$ and $|A| = k$ then
\begin{align}
\IP_S (A) = \frac{1}{N + 1}\; \frac{1}{\binom{N}{k}}
\label{eq: Ps}
\end{align}
(see \cite{Straffin}). This measure makes coalitions of the same cardinality equally likely and satisfies
\begin{align}
\IP_S (\{A ~  | ~ |A| = k \}) = \frac{1}{N+1}
\end{align}
By $D_S, D_S^+, D_S^-$ we denote the quantities $D_{\IP_S}, D_{\IP_S}^+, D_{\IP_S}^-$.\\
$\IP_S$ can be written as
\begin{align}
\IP_S \big((x_1, \cdots ,x_N)\big) = \int_0^1 p^{\,\sum x_i}\; (1 - p)^{N - \sum x_i} ~  dp
\label{eq:Psmu}
\end{align}

\item The unanimity measure
\begin{align}
\IP_u (A) = \begin{cases} ~\frac{1}{2} & \text{if $A = \emptyset$ or $A = V$} \\ ~0 & \text{otherwise} \end{cases}
\end{align}
\item The common belief measure generalizes all three previous examples. Suppose $\mu$ is a probability measure on $[0,1]$
 (and the Borel $\sigma$-algebra) such that
 \begin{align}\label{eq:reflect}
    \mu ([\frac{1}{2} + a, \frac{1}{2} + b]) = \mu ([\frac{1}{2} - b, \frac{1}{2} - a])
 \end{align}
  then the measure
\begin{align}
\IP_{\mu} \big((x_1, \cdots ,x_N)\big)~ =~ \int_0^1 p^{\,\sum x_i}\; (1 - p)^{N - \sum x_i} ~  d\mu (p)
\label{eq:Pmu}
\end{align}
is a voting measure. We call it the common belief voting measure with CB-measure $\mu$ (see \cite{KirschPen, KirschBer} for a discussion of the Common Belief Model).

Since we have more to say about the common belief measure we introduce another way to write it
which will be convenient in later sections.

We denote by $P^{1}_{p}$ the probability measure on $\{0,1\}$ defined by $P^{1}_{p}(1)=p$ and
$P^{1}_{p}(0)=1-p$ with $0\le p \le 1$ and by $P^{N}_{p}$  the $n$-fold product of $P^{1}_{p}$
on $\{0,1\}^{N}$. Thus
\begin{align}
   P_{p}^{N}\big((x_1,x_2,\ldots,x_{N})\big)~=~p^{\,\sum x_i}\,\big(1-p\big)^{N-\sum x_i}
\end{align}
Whenever $N$ is clear from the context we write $P_{p}$ instead of
$P^{N}_{p}$.

With this notation \eqref{eq:Pmu} reads

\begin{align}
\IP_{\mu} \big(A\big)~ =~ \int_0^1 P^{N}_{p}\big(A\big) ~  d\mu (p)~ =~ \int_0^1 P_{p}\big(A\big) ~  d\mu (p)
\label{eq:Pmu1}
\end{align}
for all $A\subset\{0,1\}^{N}$. 

The Penrose-Banzhaf measure correspond to the choice $\mu = \delta_{\frac{1}{2}}$, the unanimity measure to $\mu = \frac{1}{2} \delta_0 + \frac{1}{2} \delta_1$ and the Shapley-Shubik measure to the uniform distribution (= Lebesgue measure) on $[0,1]$.\\
\end{enumerate}
\end{examples}
Instead of looking at the decisiveness of a voter one might define the influence of a voter by considering the probability that the outcome of the voting coincides with the voter's opinion.
\begin{definition}
Suppose $(V, \mathcal{V})$ is a voting system and $\IP$ a voting measure on $V$.\\
We call the probability
\begin{align}
S_{\IP}^+ (v) = \IP (\{ A \in \mathcal{V} \mid v \in A \} )
\end{align}
the $\emph{rate of affirmative success}$ of the voter $v$ (w.r.t. $\IP$).\\
Similarly,
\begin{align}
S_{\IP}^- (v) = \IP ( \{ A \notin \mathcal{V} \mid v \notin A \} )
\end{align}
is called the \emph{rate of blocking success}.\\
The quantity
\begin{align}
S_{\IP} (v) = S_{\IP}^+ (v) + S_{\IP}^- (v)
\end{align}
is called the (total) \emph{rate of success} of $v$.\\
For the Penrose-Banzhaf measure the rate of success does not give new information because
\begin{align}
S_{\IP_B} (v) = \frac{1}{2} + \frac{1}{2} D_{B} (v)
\label{eq: DS}
\end{align}
This equation goes back to $\cite{DubeyS}$.\\
Equation $\eqref{eq: DS}$ is not true for other voting measures, in fact it is $\emph{only}$ true for the Penrose-Banzhaf measure $\cite{Laruelle}$.\\
We introduce a final quantity for this section, namely the `efficiency' of a voting system, also called the `power of a collectivity to act'. It goes back to Coleman $\cite{Coleman}$ who introduced it in connection with the Penrose-Banzhaf measure.
\end{definition}
\begin{definition}
If $(V,\mathcal{V})$ is a voting system and $\IP$ a voting measure on $V$ then
\begin{align}
E_{\IP} := \IP (\mathcal{V} )
\end{align}
is called the $\emph{efficiency}$ of the voting system.
\end{definition}
\section{Permutation Invariant Voting Systems}
In this section we classify voting systems and voting measures which are invariant under
permutations of the voters.
\begin{definition}
    If $\pi$ is a permutation (=bijective mapping) on $V$ and $A\subset V$, then
    $\pi^{-1}A:=\{v\in V\mid \pi(v)\in A\}$.

   We call a voting system $(V,\sV)$ \emph{permutation invariant} (or \emph{invariant} for short)
   if for any permutation $\pi$, $A\in\sV$ implies $\pi^{-1}(A)\in \sV$.
\end{definition}
Invariant voting systems are easy to characterize: They obey the rule "One person, one vote!".
\begin{proposition}
   Every permutation invariant voting system $(V,\sV)$ is a weighted voting system. The weights can be chosen to be equal to $1$ for all voters in $V$.
\end{proposition}
\begin{proof}
  If coalitions $A$ and $B$ in $V$ contain the same number of voters, then there is a permutation
  on $V$ that maps $A$ bijective onto $B$. It follows that $A\in\sV$ if and only if $B\in\sV$.
  In other words, whether $A$ is winning depends only on the cardinality $|A|$ of $A$.

  Denote by $q$ the smallest number such that $|A|=q$ implies $A\in\sV$. Then, by monotonicity of
  $\sV$, $|B|\geq q$ implies $B\in\sV$. Since $q$ is the smallest such number $|B|<q$ implies
  $B\not\in \sV$.

  Thus $(V,\sV)$ is a weighted voting system with weights $w(v)\equiv 1$ and quota $q$.
\end{proof}
\begin{definition}
   Suppose $V$ is a finite set. A measure $\IP$ on $V$ is called \emph{permutation invariant} or
   \emph{exchangeable} if $\IP(A)=\IP(\pi^{-1}A)$ for each $A\subset V$ and each permutation $\pi$
   on $V$.
\end{definition}
All voting measures introduced in Example \ref{ex:vm} are exchangeable.

Since we are interested in the behavior of quantities like power indices and success rates
for \emph{large} voting system, we concentrate on
voting measures which can be extended to arbitrary large sets in a natural way, i.\,e. such that
the extension is still exchangeable.

If the set $V$ of voters has $N$ elements we may set $V=\{1,2,\ldots,N\}$ without loss of generality
and consider a voting measure as a measure on $\{0,1\}^{N}$ as in \eqref{eq:pv01}.

\begin{definition}
   We call an exchangeable measure $\IP$ on $\{0,1\}^{N}$ \emph{extendable}
   if for every $N'>N$ there is an
   exchangeable measure $\IP'$ on $\{0,1\}^{N'}$ such that $\IP$ is the restriction of $\IP'$
   on $\{0,1\}^{N}$.
\end{definition}
The voting measures of Example \ref{ex:vm} are extendable.

\begin{theorem}\label{th:cbm}
   Suppose $\IP$ is an exchangeable and extendable voting measure on $V=\{1,2,\ldots,N\}$ then
   $\IP$ is a common belief measure (see Example \ref{ex:vm}.4),\\
   i.\,e. there is a measure $\mu$ on $[0,1]$ with \eqref{eq:reflect} such that
   \begin{align}\label{eq:cbm}
\IP_{\mu} \big(A\big)~ =~ \int_0^1 P_{p}\big(A\big) ~  d\mu (p)
 \end{align}
\end{theorem}
Theorem \ref{th:cbm} is a version of the celebrated theorem of de Finetti (\cite{deFin}). De Finetti's
theorem can be found at various places and in various formulations, see e.\,g. \cite{Aldous} or
\cite{Klenke}. For an introduction and an elementary proof see \cite{KirschMom}.

\begin{proof}
 Kolmogorov's extension theorem ensures that there is an exchangeable measure $\widetilde{\IP}$ on $\{0,1\}^{\IN}$  whose restriction on $\{0,1\}^{N}$ is given by $\IP$.

 By de Finetti's theorem $\widetilde{\IP}$ and therefore $\IP$ have the structure \eqref{eq:cbm}.
 The property \eqref{eq:reflect} follows from the assumption that $\IP$ is a voting measure.
\end{proof}

\section{Penrose-Banzhaf vs. Shapley-Shubik: A case study}
In this section we consider the behavior of efficiency, decisiveness and rate of success in simple voting systems under the Penrose-Banzhaf and the Shapley-Shubik measure.
Our first result is
\begin{proposition}
\label{prop:desm} Let $(V,\mathcal{V})$ be a simple majority voting system with $N$ voters then
\begin{enumerate}
\item ~\vspace*{-1cm}
\begin{align}\label{eq:DB}
  &D_B (v) \approx \frac{2}{\sqrt{2 \pi}} \cdot \frac{1}{\sqrt{N}} \text{~  as ~} N \rightarrow \infty \qquad\qquad\text{and}\\
     &E_B \approx \frac{1}{2} \text{~ as ~} N \rightarrow \infty
\end{align}
\item ~\vspace*{-1cm}
\begin{align}\label{eq:SB}
 &D_S (v) = \frac{1}{N} \text{~ for all ~} N \qquad\qquad\text{and}\\
  &E_S \approx \frac{1}{2} \text{~ as ~} N \rightarrow \infty
\end{align}
\end{enumerate}
\end{proposition}
\begin{notation}
By $a_N \approx  b_N$ we mean $\frac{a_N}{b_N} \rightarrow 1$ as $N \rightarrow \infty$.
\end{notation}
\begin{proof}
The proof of $1.$ is quite standard, see for example $\cite{Felsenthal}$.\\
$2.$ follows from the fact that $\sum_{v \in V} D_S (v) = 1$ and $D_S (v) = D_S (v^{'})$ for all $v, v^{'}$.
\end{proof}
\begin{remark}
A calculation shows that for odd $N$
\begin{align}
D_S^- (v) = D_S^+ (v) = \frac{1}{2} \frac{1}{N}
\end{align}
and for even $N$
\begin{align}
D_S^- (v) = \frac{1}{2} \frac{1}{N} - \frac{1}{2 N (N+1)}\\
D_S^+ (v) = \frac{1}{2} \frac{1}{N} + \frac{1}{2 N (N+1)}
\end{align}
Proposition $\ref{prop:desm}$ has an immediate consequences for the success rate of voters. From $\eqref{eq: DS}$ we infer that
\begin{align}
S_B (v) \approx \frac{1}{2} + \frac{1}{\sqrt{2 \pi}} \frac{1}{\sqrt{N}}
\label{eq: SBsm}
\end{align}
for simple majority voting systems.\\
As one might expect the Penrose-Banzhaf power goes to zero as $N$ increases and the success rate goes to $\frac{1}{2}$, the success rate of a dummy player.\\
The Shapley-Shubik power goes to  zero as $N \rightarrow \infty$ as well, in fact, even faster than the Penrose-Banzhaf power (see Proposition \ref{prop:desm}).

It may be somewhat surprising that the Shapley-Shubik success rate does $\emph{not}$ go to $\frac{1}{2}$, but rather stays at about $\frac{3}{4}$ independent of the size of $V$. We will prove this fact in greater generality below.\\
Now, we turn to simple voting systems with a qualified majority, i.e. we consider weighted voting systems with weights $w(v) = 1$ and arbitrary relative quota $r$.\\
First, we look at the behavior of the efficiency for fixed $r$ and $N$ large.
\end{remark}
\begin{theorem}
$\label{th: effB}$
Suppose $(V, \mathcal{V})$ is a weighted system with $N$ voters, with weights $w (v) = 1$ \;for all $v \in V$ and relative quota $r$. Then
\begin{enumerate}
\item $E_B \rightarrow
\begin{cases}\quad 1 & \text{for $r < \frac{1}{2}$} \\ \quad\frac{1}{2} & \text{for $r = \frac{1}{2}$} \\ \quad
0 & \text{for } r > \frac{1}{2} \end{cases}$

\item For $r > \frac{1}{2}$
\begin{align}
E_B \leq e^{-2 (r - \frac{1}{2})^2 N}
\end{align}
\end{enumerate}
\end{theorem}
This theorem tells us that the (Banzhaf-)efficiency of a voting game goes extremely fast to zero if the voting body is enlarged and the relative quota is kept fixed at $r > \frac{1}{2}$. This is exactly what happened for the Council of the European Union during EU enlargements!

\begin{proof}
Part 1. follows from the strong law of large numbers $\cite{Klenke}$ and Proposition $\ref{prop:desm}$.

 2. is an application of Hoeffding's inequality (see the Appendix).

\end{proof}
In contrast to the above result, the efficiency according to Shapley-Shubik does not go to zero for $r>1/2$.
\begin{theorem}
$\label{th: effS}$
Suppose $(V, \mathcal{V})$ is a weighted voting system with $N$ voters, weights $w (v) = 1$ for all $v \in V$ and relative quota $r$.
\begin{align}
E_S \rightarrow (1 - r) \text{~ as ~} N \rightarrow \infty.
\end{align}
\end{theorem}
\begin{proof}
From $\eqref{eq:Pmu1}$ we infer
\begin{align}
E_{S}~ &=~ \IP_S (\sum X_i \geq rN)\\
&=~\int_0^1   P_p (\frac{1}{N} \sum X_i \geq r)\;dp
\label{eq:Psk}
\end{align}
The expression under the integral in $\eqref{eq:Psk}$
\begin{align}
P_p (\frac{1}{N} \sum X_i \geq r)
\end{align}
is the probability with respect to $P_{p}$ that the arithmetic mean of the $X_{i}$  is not less
than $r$. The random variables $X_{i}$ are independent under the measure $P_{p}$.
Thus we may apply the law of large numbers to show that this expression goes to $0$ for $r > p$ and to $1$ for $r < p$. Consequently $\eqref{eq:Psk}$ converges to
\begin{align}
\int^1_r dp = 1 - r
\end{align}
\end{proof}
We turn to an investigation of the success rate.

Before we consider the case of arbitrary $r$ we discuss in detail the case $r = \frac{1}{2}$.\\
It is quite obvious that for simple majority systems
\begin{align}
S_B^+ (v) \rightarrow \frac{1}{4} \text{~ and ~} S_B^- (v) \rightarrow \frac{1}{4} \text{~ as ~} N \rightarrow \infty
\end{align}
The following result about $S_S^+$ and $S_S^-$ is perhaps not so obvious.
\begin{theorem}
$\label{th: SS12}$
Let $(V, \mathcal{V})$ be a simple majority voting system with $N$ voters, then
\begin{align}
S_S^+ (v) = \begin{cases} \frac{3}{8} + \frac{1}{8} \frac{1}{N} & \text{for odd $N$} \\ \frac{3}{8} - \frac{1}{8} \frac{1}{N+1} & \text{for even $N$} \end{cases}
\end{align}
\begin{align}
S_S^- (v) = \begin{cases} \frac{3}{8} + \frac{1}{8} \frac{1}{N} & \text{for odd $N$} \\ \frac{3}{8} + \frac{3}{8} \frac{1}{N+1} & \text{for even $N$} \end{cases}
\end{align}
Consequently
\begin{align}
S_S (v) = \begin{cases} \frac{3}{4} + \frac{1}{4} \frac{1}{N} & \text{for odd $N$} \\ \frac{3}{4} + \frac{1}{4} \frac{1}{N+1} & \text{for even $N$} \end{cases}
\end{align}
In particular
\begin{align}
S_S (v) \approx \frac{3}{4} \text{~ as ~} N \rightarrow \infty
\end{align}
\end{theorem}
\begin{proof}
We may assume that $V=\{1,2,\ldots,N\}$ and $v=1$.
Let us start with $N$ odd, say $N = 2n + 1$.\\
Then
\begin{align}
S_S^+ (1) & = \IP_S (x_1 = 1, \sum_{i=2}^N x_i \geq n) \notag\\
& = \int^1_0 P_p (x_1 = 1, \sum_{i=2}^N x_i \geq n) ~  dp \notag\\
& = \int^1_0 p \cdot P_p (\sum_{i=2}^N x_i \geq n) ~  dp \notag\\
& = \sum_{k=n}^{N-1} \binom{N-1}{k} \int_0^1 p^{k+1} (1-p)^{N - (k+1)} ~  dp \notag\\
& = \frac{1}{N+1} \sum^{N-1}_{k=n} \frac{\binom{N-1}{k}}{\binom{N}{k+1}} = \frac{1}{N+1} \sum_{k=n}^{N-1} \frac{k+1}{N} \notag\\
& = \frac{1}{N(N+1)} \sum_{k=n+1}^{N} k = \frac{1}{N(N+1)} \frac{1}{2} \big(N(N+1) - n(n+1)\big) \notag\\
& = \frac{1}{2} - \frac{1}{2} \frac{n(n+1)}{N(N+1)} = \frac{1}{2} - \frac{1}{8} \frac{(N-1)(N+1)}{N(N+1)} \notag\\
& = \frac{3}{8} + \frac{1}{8} \frac{1}{N}
\end{align}
\begin{align}
S_S^- (1) & = \IP_S (x_1 = 0, \sum_{i=2}^N x_i \leq n) \notag \\
& = \sum_{k=0}^{n} \binom{N-1}{k} \int_0^1 (1-p) p^{k} (1-p)^{N - k - 1} ~ dp \notag\\
& = \sum_{k=0}^{n} \binom{N-1}{k} \int_0^1 p^{k} (1-p)^{N - k} ~ dp \notag\\
& = \frac{1}{N+1} \sum^{n}_{k=0} \frac{\binom{N-1}{k}}{\binom{N}{k}} \notag\\
& = \frac{1}{N(N+1)} \sum_{k=0}^n (N-k) = \frac{1}{N(N+1)} \big((n+1)N - \frac{1}{2} n (n+1)\big) \notag\\
& = \frac{3}{8} + \frac{1}{8} \frac{1}{N}
\end{align}
Thus for $N$ odd we have
\begin{align}
S_S^+ (1) = S_S^- (1) = \frac{3}{8} + \frac{1}{8} \frac{1}{N}\\
\text{and ~} S_S (1) = \frac{3}{4} + \frac{1}{4} \frac{1}{N} \rightarrow \frac{3}{4}
\end{align}
The calculation for even $N$ goes along the same lines.
\end{proof}
\begin{theorem}
Suppose $(V, \mathcal{V})$ is a weighted voting system with $N$ voters, weights $w (v) = 1$ for all $v \in V$ and relative quota $r$.
\begin{enumerate}
\item
\begin{align}
S_B (v) \rightarrow \frac{1}{2} \text{~  as ~} N \rightarrow \infty
\end{align}
For $r > \frac{1}{2}$
\begin{align}\label{eq: SBP}
S_B^+ (v) \leq \frac{1}{2} e^{-2 (r - \frac{1}{2})^2 (N - 1)}
\end{align}
\item
\begin{align}
S_S^+ (v) & \rightarrow \frac{1}{2} - \frac{1}{2} r^2 \text{~  as ~} N \rightarrow \infty\\
S_S^- (v) & \rightarrow \frac{1}{2} - \frac{1}{2} (1 - r)^2 \text{~  as ~} N \rightarrow \infty
\end{align}
Consequently
\begin{align}
S_S (v) \rightarrow 1 - \frac{1}{2} (r^2 + (1 - r)^{2}) \text{~  as ~} N \rightarrow \infty
\end{align}
\end{enumerate}
\end{theorem}
\begin{remark}
$S_S (v) \approx 1 - \frac{1}{2} (r^2 + (1 - r)^2)$ is biggest for $r = \frac{1}{2}$ where it equals $\frac{3}{4}$ and smallest for $r=0$ and $r=1$ where it is $\frac{1}{2}$.
\end{remark}
\begin{proof}

\begin{enumerate}
\item
\begin{align}
S_B^+ (1) & = \IP_B (x_1 = 1, \sum_{i=2}^N x_i \geq rN-1)\\
& = \frac{1}{2} \IP_B ( \sum_{i=2}^N x_i \geq rN-1)
\end{align}
since under $\IP_B$ the $x_i$ are independent.\\
Another application of Hoeffding's inequality (Theorem \ref{th:Hoeffding}) gives $\eqref{eq: SBP}$.\\
Similarly
\begin{align}
S_B^- (1) = \frac{1}{2} \IP_B (\sum_{i=2}^N x_i < rN) \notag
\end{align}
so
\begin{align}
S_B (1) & = \frac{1}{2} (1 + \IP_B (\sum_{i=2}^N x_i \in [rN-1, rN))) \notag\\
& \leq \frac{1}{2} + C \frac{1}{\sqrt{N}}
\end{align}
\item A computation as in the proof of Theorem $\ref{th: SS12}$ shows that
\begin{align}
\IP_S (x_1 = 1, \sum_{i=2}^N x_i \geq M) \notag\\
= \frac{1}{2} - \frac{1}{2} \frac{M(M+1)}{N(N+1)}
\end{align}
\end{enumerate}
If we insert $M = \lceil rN\rceil - 1$, where $\lceil x\rceil$ is the smallest integer not less than $x$, we obtain
\begin{align}
S_S^+ (1) & = \frac{1}{2} - \frac{1}{2} \frac{(\lceil rN\rceil	- 1) \lceil rN\rceil}{N(N+1)} \notag\\
& \rightarrow \frac{1}{2} - \frac{1}{2} r^2 \text{~  as ~} N \rightarrow \infty
\end{align}	
Similarly
\begin{align}
S_S^- (1) & = \IP_S (x_1 = 0, \sum_{i=2}^N x_i \leq \lceil rN\rceil - 1) \notag\\
& = \IP_S (x_1 = 1, \sum_{i=2}^N x_i \geq N - \lceil rN\rceil) \notag\\
& = \frac{1}{2} - \frac{1}{2} \frac{(N - \lceil rN\rceil) (N - \lceil rN\rceil + 1)}{N(N+1)} \notag\\
& \rightarrow \frac{1}{2} - \frac{1}{2} (1-r)^2
\end{align}
\end{proof}

\section{Weighted Voting Systems and the Common Belief Model}
We turn to our most general case.\\
In this section we consider large weighted voting system. More precisely, we consider a sequence $\{w_n\}_{n \in N}$ of non negative real numbers and for each $N$ the voting system with weights $w_1, \cdots , w_N$ and (fixed) relative quota $r$, thus we have voting systems $(V_N, \mathcal{V_N})$ with $V_N = \{1, \cdots , N\}$ and $A \in \mathcal{V}_N$ if and only if
\begin{align}
\sum_{i \in A} w_i \geq r \sum_{i \in V_N} w_i
\end{align}
To shorten notation we write $w(A) = \sum_{i \in A} w_i$ and $w (N) = \sum_{i=1}^N w_i$.\\
On $(V_N, \mathcal{V}_N)$ we consider the voting measure
\begin{align}
\IP_{\mu} (A) = \int_0^1 p^{|A|} (1 - p)^{N - |A|} ~  d \mu (p)
\end{align}
for $A \subset V_N$ and a measure $\mu$ on $[0,1]$ with $\mu ([\frac{1}{2} + a, \frac{1}{2} + b]) = \mu ([\frac{1}{2} - a, \frac{1}{2} - b])$.
\begin{definition}
We define the Laakso-Taagepera index of the sequence $\{w_n\}$ by
\begin{align}
LT_N = \frac{\sum_{n=1}^N w_n^2}{(\sum_{n=1}^N w_n)^2}
\end{align}
The Laakso-Taagepera index is named after $\cite{LaaksoT}$.\\
We start with a result of Langner $\cite{Langner}$.
\end{definition}
\begin{theorem}
$\label{th: Lang}$
If $LT_N \rightarrow 0$ and $\mu (\{r\}) = 0$ then the efficiency $E_N$ of the voting systems $(V_N, \mathcal{V_N})$ satisfies
\begin{align}
E_N \rightarrow \mu ([r,1])
\end{align}
\end{theorem}
For the reader's convenience we reprove this theorem.

\begin{proof}
\begin{align}
   E_{N}~=~\int_{0}^{1} P_{p}\big(\sum_{i=1}^{N} w_{i} X_{i} \geq r \sum_{i=1}^{N} w_{i}\big)\;d\mu(p)
\end{align}
By Corollary \ref{cor:Hoeffding} the integrand converges to $0$ for $r>p$ and to $1$ for $r<p$,
hence
\begin{align}
   E_{N}~\to~\int_{0}^{1} \chi_{\{p>r\}}(p)\,d\mu(p)~=~\mu([r,1])\,.
\end{align}
where
\begin{align*}
   \chi_{p>r}(p)~=~\left\{
                     \begin{array}{ll}
                       1, & \hbox{if $p>r$;} \\
                       0, & \hbox{otherwise.}
                     \end{array}
                   \right.
\end{align*}
We estimate
\begin{align}
& P_p (| \sum w_i x_i - p \sum w_i| \geq \alpha \sum w_i) \notag\\
& \leq \frac{1}{\alpha^2} \frac{\sum_{i} E_p ((\sum w_i (x_i - p))^2)}{(\sum w_i)^2} \notag\\
& = \frac{1}{\alpha^2} p (1 - p) \frac{\sum_i w_i^2}{(\sum w_i)^2} = \frac{p (1 - p)}{\alpha^2} LT_N
\end{align}
It follows that
\begin{align}
E_N = \int_0^1 P_p (\sum w_i x_i \geq r \sum w_i) ~  d \mu (p)
\end{align}
converges to
\begin{align}
\int^1_0 \chi (p > r) ~  d \mu (p) = \mu ([r,1])
\end{align}
\end{proof}
In a similar way we can compute the rate of success in such systems.
\begin{theorem}
$\label{th: SucCMB}$
If $LT_N \rightarrow 0$ and $\mu (\{r\}) = 0$ then the rate of success with respect to $\IP_{\mu}$ satisfies
\begin{align}
S_{\IP_{\mu}}^+ (v) & \rightarrow \int_r^1 p ~  d \mu (p)\\
S_{\IP_{\mu}}^- (v) & \rightarrow \int_{r-1}^1 p ~  d \mu (p)
\end{align}
\end{theorem}
\begin{proof}
Without loss we compute the rate of success for voter $i=1$.
\begin{align}
S_{\IP_{\mu}}^+ (1) & = \int_0^1 P_p (x_1 = 1, \sum_{i=2}^N w_i x_i \geq r \sum w_i - w_1) ~  d \mu (p) \notag\\
& = \int_0^1 p P_p (\sum_{i=2}^N w_i x_i \geq r \sum w_i - w_1) ~  d \mu (p) \notag\\
& \rightarrow \int_r^1 p ~  d \mu (p)
\end{align}
\end{proof}
\appendix
\section{Appendix}
For the reader's convenience, in this appendix we present a few results needed in the main text.
In particular, we formulate Hoeffding's inequality.

\begin{theorem}[Hoeffding's Inequality]\label{th:Hoeffding}
 Suppose $X_{i}, i=1,\ldots,N$ are independent random variables such that $X_{i}\in[a_{i},b_{i}]$ almost surely.

 Set $\sigma^2=\sum_{i=1}^{N} (b_{i}-a_{i})^2$. Then
 {\large\begin{align}\label{eq:Hoeffding}
   &\IP\Big(\sum_{i=1}^{N} X_{i}~\geq~\IE(X_{i}) +\lambda\Big)~
   \leq~\,e^{-2\frac{\lambda}{\sigma^2}}\qquad   \text{and}\\
   &\IP\Big(\sum_{i=1}^{N} X_{i}~\leq~\IE(X_{i}) -\lambda\Big)~
   \leq~\,e^{-2\frac{\lambda}{\sigma^2}}
 \end{align}}
\end{theorem}

For a proof of Theorem \ref{th:Hoeffding} see e.\,g. \cite{Pollard}.

An immediate consequence of Hoeffding's inequality is the following proposition.
As before $P_{p}$ with $0\leq p\leq 1$ denotes the probability measure on $\{0,1\}^{N}$ given by:
{\large\begin{align}
   P_{p}\big(x_{1},x_{2},\ldots,x_{N}\big)~=~p^{\,\sum x_{i}}\,\big(1-p\big)^{N-\sum x_{i}}
\end{align}}
and $E_{p}$ denotes expectation with respect to $P_{p}$.

\begin{proposition}\label{prop:Pweight}
   Let $X_{i}, i=1,\ldots,N$ be random variables with distribution $P_{p}$ and $w_{1},\ldots, w_{N}\in
   [0,\infty)$, then for $\lambda\geq 0$ {\large
   \begin{align}\label{eq:Pweight}
      P_{p}\Big(\big|\sum_{i=1}^{N}w_{i} X_{i}-p \sum_{i=1}^{N}w_{i}\big|\geq \alpha\sum_{i=1}^{N} w_{i}\Big)
      \leq 2\,e^{{-2\alpha^{2}\, \frac{(\sum w_{i})^{2}}{\sum w_{i}^{2}}}}
   \end{align}}
\end{proposition}
\begin{proof}
   The random variables $Y_{i}=w_{i}X_{i}$ are independent (under $P_{p}$) and take values in $[0,w_{i}]]$. Moreover $E_{p}(Y_{i})=p w_{i}$. Thus \eqref{eq:Pweight} follows from
   Theorem \ref{th:Hoeffding}.
\end{proof}

\begin{corollary}\label{cor:Hoeffding}
   Suppose the Laakso-Taagepera index $LT_{N}=\frac{\sum w_{i}^{2}}{(\sum w_{i})^{2}}$ goes to zero
   as $N\to\infty$ then
   \begin{align}
      \text{If } \alpha>p\quad &P_{p}\Big(\sum_{i=1}^{N}w_{i} X_{i}~\geq \alpha\,\sum_{i=1}^{N}w_{i}\Big)~\to~0\qquad\text{as }~N\to \infty\\
      \text{If } \alpha<p\quad &P_{p}\Big(\sum_{i=1}^{N}w_{i} X_{i}~\leq \alpha\,\sum_{i=1}^{N}w_{i}\Big)~\to~0\qquad\text{as }~N\to \infty
   \end{align}
\end{corollary}

\end{document}